\documentclass[leqno,12pt]{article} 
\usepackage{amsmath, amssymb}
\usepackage{amsthm} 
\usepackage{pictexwd,dcpic}
\usepackage{xfrac}
\usepackage{xypic}

\theoremstyle{plain} 
\newtheorem{theorem}{\indent\sc Theorem}[section]
\newtheorem{lemma}[theorem]{\indent\sc Lemma}
\newtheorem{corollary}[theorem]{\indent\sc Corollary}
\newtheorem{proposition}[theorem]{\indent\sc Proposition}

\theoremstyle{definition} 

\newtheorem{remark}[theorem]{\indent\sc Remark}
\newtheorem{example}[theorem]{\indent\sc Example}
\newtheorem{notation}[theorem]{\indent\sc Notation}

\numberwithin{equation}{section}

\def\Z{\zeta}
\def\V{\varphi}
\def\O{\omega}

\def\G{\Gamma}
\def\Q{\mathbb{Q}}

\def\o{\mathcal{O}}
\def\F{\mathfrak{f}}

\makeatletter
\def\address#1#2{\begingroup
\noindent\parbox[t]{7.8cm}{%
\small{\scshape\ignorespaces#1}\par\vskip1ex
\noindent\small{\itshape E-mail address}%
\/: #2\par\vskip4ex}\hfill%
\endgroup}%

\makeatother
\title{{Construction of class fields over imaginary biquadratic fields}} 
\author{
\textsc{Ja Kyung Koo and Dong Sung Yoon$^*$} 
}

\date{} 
\begin{document}

\maketitle

\footnote{$^*$Corresponding author.}


\begin{abstract}
Let $K$ be an imaginary biquadratic field and $K_1$, $K_2$ be its imaginary quadratic subfields.
For integers $N>0$, $\mu\geq 0$ and an odd prime $p$ with $\gcd(N,p)=1$, let $K_{(Np^\mu)}$ and $(K_i)_{(Np^\mu)}$ for $i=1,2$ be the ray class fields of $K$ and $K_i$, respectively, modulo $Np^\mu$.
We first present certain class fields $\widetilde{K_{N,p,\mu}^{1,2}}$ of $K$, in the sense of Hilbert, which are generated by Siegel-Ramachandra invariants of $(K_i)_{(Np^{\mu+1})}$ for $i=1,2$ over $K_{(Np^\mu)}$ and show that 
$K_{(Np^{\mu+1})}=\widetilde{K_{N,p,\mu}^{1,2}}$ for almost all $\mu$.

\end{abstract}

\maketitle

\section{Introduction}

In 1900 Hilbert asked at the Paris ICM, as his 12-th problem, that what kind of analytic functions and algebraic numbers are necessary to construct all abelian extensions of given number fields.
For any number field $K$ and a modulus $\mathfrak{m}$ of $K$, it is well known that there is a unique maximal abelian extension of $K$ unramified outside $\mathfrak{m}$ with certain ramification condition (\cite{Takagi} or \cite[Theorem 8.6]{Cox}), which we call the \textit{ray class field} of $K$ modulo $\mathfrak{m}$.
Since any abelian extension of $K$ is contained in some ray class field modulo $\mathfrak{m}$ by class field theory, in order to approach the problem we are first in need of constructing ray class fields of given number fields as Ramachandra did in \cite{Ramachandra} over imaginary quadratic fields.
\par
Historically, over imaginary quadratic fields, after Hasse (\cite{Hasse}) one can now construct it by using the 
theory of complex multiplication of elliptic curves together with the singular values of modular functions and elliptic functions (\cite[Chapter 10]{Lang} or \cite{Shimura-5}). 
However, for any other number fields we know relatively little until now.
For example, over cyclotomic field see \cite{Komatsu}, \cite{K-Y}, \cite{Ribet} (also \cite{Mazur}), \cite{Shimura-1}, and   over biquadratic fields we refer to \cite{Azizi}, \cite{Bae}, \cite{Benjamin}, \cite{Buell}, \cite{Sime} and \cite{Yokoi}.
\par
In this paper we will concentrate on the case of imaginary biquadratic fields $K$.
There are two imaginary quadratic subfields $K_1$, $K_2$ in $K$.
For integers $N>0$, $\mu\geq 0$ and an odd prime $p$ with $\gcd(N,p)=1$, we denote by $K_{(Np^\mu)}$ and $(K_i)_{(Np^\mu)}$ for $i=1,2$ the ray class fields of $K$ and $K_i$ modulo $Np^\mu$, respectively.
In Section \ref{Shimura's reciprocity law on imaginary biquadratic fields}, we apply Shimura's reciprocity law to the field $K$ and put emphasis on the necessity of using modular units to construct class fields over $K$ as Hilbert proposed.
On the other hand, Jung-Koo-Shin (\cite[Theorem 3.5]{Jung}) constructed relatively simple ray class invariants over imaginary quadratic fields by means of the singular values of certain Siegel functions, namely, Siegel-Ramachandra invariants.
Using this idea we define in Section \ref{(I)} and \ref{(II)} certain class field $\widetilde{K_{N,p,\mu}^{1,2}}$ of $K$ which is generated by Siegel-Ramachandra invariants of $(K_i)_{(Np^{\mu+1})}$ for $i=1,2$ over $K_{(Np^\mu)}$, and provide a necessary and sufficient condition for the field $\widetilde{K^{1,2}_{N,p,\mu}}$ to become $K_{(Np^{\mu+1})}$ (Theorem \ref{K_{1,2}} and \ref{K_{1,2}-2}).

\begin{notation}
For $z\in\mathbb{C}$, we denote by $\overline{z}$ the complex conjugate of $z$ and put $e(z)=e^{2\pi i z}$.
If $R$ is a ring with identity and $r,s\in\mathbb{Z}_{>0}$, $M_{r\times s}(R)$ indicates the ring of all $r\times s$ matrices with entries in $R$. 
In particular, we set $M_{r}(R)=M_{r\times r}(R)$.
The identity matrix of $M_{r}(R)$ is written by $1_r$ and the transpose of a matrix $\alpha$ is written as ${^t}\alpha$. 
And, $R^\times$ means the group of all invertible elements of $R$.
If $G$ is a group and $g_1,g_2,\ldots,g_r$ are elements of $G$, let $\langle g_1,g_2,\ldots,g_r \rangle$ be the subgroup of $G$ generated by $g_1,g_2,\ldots,g_r$.
For a number field $K$, let $\o_K$ be the ring of integers of $K$.
If $a\in\o_K$ we denote by $(a)$ the principal ideal of $K$ generated by $a$.
When $\mathfrak{f}$ is an integral ideal of $K$, we mean by $\mathcal{N}(\mathfrak{f})$ the absolute norm of an ideal $\mathfrak{f}$.
For a finite extension $L$ of $K$, let $[L:K]$ be the extension degree of $L$ over $K$.
When $L/K$ is abelian, we mean by $\big(\frac{L/K}{\cdot}\big)$ the Artin map of $L/K$.
Further, we let $\Z_N=e^{2\pi i/N}$ be a primitive $N$-th root of unity for a positive integer $N$.
\end{notation}

\section{Shimura's reciprocity law on imaginary biquadratic fields}\label{Shimura's reciprocity law on imaginary biquadratic fields}
In this section we shall briefly recall Siegel modular forms and apply Shimura's reciprocity law to an imaginary biquadratic field.
\par
Let $n$ be a positive integer and 
\begin{equation*}
J=J_n=
\left[\begin{matrix}
0&-1_n\\
1_n&0
\end{matrix}\right].
\end{equation*}
For a commutative ring $R$ with unity, we let
\begin{equation}\label{symplectic}
\begin{array}{rcl}
\mathrm{GSp}_{2n}(R)&=&\big\{\alpha\in \mathrm{GL}_{2n}(R)~|~{^t}\alpha J\alpha=\nu(\alpha)J ~\textrm{ with $\nu(\alpha)\in R^\times$}\big\},\\
\mathrm{Sp}_{2n}(R)&=&\big\{\alpha\in \mathrm{GSp}_{2n}(R)~|~\nu(\alpha)=1 \big\}.
\end{array}
\end{equation}
We set $G_{\Q}=\mathrm{GSp}_{2n}(\mathbb{Q})$ and $G_{\Q+}=\{\alpha\in G_\Q~|~\nu(\alpha)>0 \}$.
Let 
\begin{equation*}
\mathbb{H}_n=\{Z\in M_n(\mathbb{C})~|~{^t}Z=Z,~\mathrm{Im}(Z)>0\}
\end{equation*}
be the \textit{Siegel upper half-space of degree $n$}.
Here, for a hermitian matrix $\xi$ we write $\xi>0$ when $\xi$ is positive definite.
An element $\alpha=
\left[\begin{matrix}
A&B\\
C&D
\end{matrix}\right]$
of $G_{\Q+}$ acts on $\mathbb{H}_n$ by
\begin{equation*}
\alpha(Z)=(AZ+B)(CZ+D)^{-1},
\end{equation*}
where $A,B,C,D\in M_n(\mathbb{Q})$.
\par
For every positive integer $N$, let 
\begin{equation*}
\G(N)=\big\{\gamma \in \mathrm{Sp}_{2n}(\mathbb{Z})~|~\gamma\equiv 1_{2n}\pmod{N\cdot M_{2n}(\mathbb{Z})} \big\}.
\end{equation*} 
For an integer $k$ a holomorphic function 
$f:\mathbb{H}_n\rightarrow\mathbb{C}$ is called a \textit{(classical) Siegel modular form of degree $n$, weight $k$ and level $N$} if
\begin{eqnarray*}
f(\gamma (Z))=\det(CZ+D)^k f(Z) 
\end{eqnarray*}
for all $\gamma=\left[\begin{matrix}A&B\\C&D\end{matrix}\right] \in \G (N)$ and $Z\in \mathbb{H}_n$, plus the requirement  that $f$ is holomorphic at every cusp when $n=1$. 
A Siegel modular form $f$ of degree $n$ and level $N$ has a Fourier expansion of the form
\begin{equation*}
f(Z)=\sum_{\xi}A(\xi)e(\mathrm{tr}(\xi Z)/N)
\end{equation*}
with $A(\xi)\in \mathbb{C}$, where $\xi$ runs over all positive semi-definite half-integral symmetric matrices of degree $n$ \cite[\S4 Theorem 1]{Klingen}. 
Here, a symmetric matrix $\xi \in \mathrm{GL}_n(\Q)$ is called \textit{half-integral} if $2\xi$ is an integral matrix 
and its diagonal entries are even.
\par
For a subfield $D$ of $\mathbb{C}$, let
\begin{equation*}
\begin{array}{rcl}
\mathcal{M}_k^n(\Gamma(N),D)&=&\textrm{the space of Siegel modular forms of weight $k$ and level $N$}\\
&&\textrm{with Fourier coefficients in $D$},\vspace{0.1cm}\\
\mathcal{M}_k^n(D)&=&\displaystyle\bigcup_{N=1}^\infty\mathcal{M}_k^n(\Gamma(N),D),\vspace{0.1cm}\\
\mathcal{A}_0^n(\Gamma(N),D)&=&\textrm{the field of functions of the form $g/h$}\\
&&\textrm{with $g\in\mathcal{M}_k^n(D)$ and $h\in\mathcal{M}_k^n(D)\setminus\{0\}$ for the same weight $k$,}\\
&&\textrm{which are invariant under $\Gamma(N)$}.
\end{array}
\end{equation*}
In particular, we set 
\begin{eqnarray*}
\mathcal{F}_N^n&=&\mathcal{A}_{0}^n\big(\G(N),\mathbb{Q}(\Z_N)\big),\\
\mathcal{F}^n&=&\bigcup_{N=1}^\infty \mathcal{F}_N^n.
\end{eqnarray*}

\par
Let $G_{\mathbb{A}}$ be the adelization of the group $G_{\mathbb{Q}}$, $G_0$ be the non-archimedean part of $G_\mathbb{A}$ and $G_\infty$ be the archimedean part of $G_\mathbb{A}$. 
Then $\nu$ in (\ref{symplectic}) defines a homomorphism $G_{\mathbb{A}}\rightarrow \mathbb{Q}_\mathbb{A}^\times$.
We put $G_{\infty+}=\{x\in G_\infty~|~\nu(x)> 0\}$ and $G_{\mathbb{A}+}=G_0G_{\infty+}$. 
For every algebraic number field $F$, let $F_\mathrm{ab}$ be the maximal abelian extension of $F$ and $F_\mathbb{A}^\times$ be the idele group of $F$. 
By class field theory, every element $x$ of $F_\mathbb{A}^\times$ acts on the field $F_\mathrm{ab}$ as an automorphism. 
We then write this automorphism as $[x,F]$. 
On the other hand, any element of $G_{\mathbb{A}+}$ acts on the space $\mathcal{A}_0^n(\mathbb{Q}_\mathrm{ab})$ as an automorphism (\cite[p.680]{Shimura-4}). 
If $x\in G_{\mathbb{A}+}$ and $f\in \mathcal{A}_0^n(\mathbb{Q}_\mathrm{ab})$, by $f^x$ we mean  the image of $f$ under $x$.
\par

For each index $i=1,2$, let $K_i=\mathbb{Q}(\sqrt{-d_i})$ be an imaginary quadratic field with square-free positive integer $d_i$ and set
\begin{displaymath}
\rho_i = \left\{ \begin{array}{ll}
\displaystyle-\frac{1}{\sqrt{-d_i}}~ & \textrm{if $-d_i\equiv 1\pmod{4}$}\vspace{0.2cm}\\
\displaystyle-\frac{1}{2\sqrt{-d_i}}~ & \textrm{if $-d_i\equiv 2,3\pmod{4}$}.
\end{array} \right.
\end{displaymath}
Note that $\rho_i$ is the number in $K_i$ for which 
$-\rho_i^2$ is totally positive, Im$(\rho_i)>0$ and Tr$_{K_i/\mathbb{Q}}(\rho_i x)\in\mathbb{Z}$ for all $x\in\mathcal{O}_{K_i}$.
Let $L_i=\mathcal{O}_{K_i}$ be a lattice in $\mathbb{C}$.
Then, for $z_i,w_i \in\mathbb{C}$ we define an $\mathbb{R}$-bilinear form $E_i(z_i,w_i)$ on $\mathbb{C}$ by
\begin{equation*}
E_i(z_i,w_i)= \rho_i(z_i\overline{w_i}-\overline{z_i}w_i).
\end{equation*}
And, $E_i$ becomes a non-degenerate Riemann form on the complex torus $\mathbb{C}/L_i$ satisfying
\begin{equation*}
E_i(\alpha_i,\beta_i)=\mathrm{Tr}_{K_i/\mathbb{Q}}(\rho_i\alpha_i\overline{\beta_i})
 \textrm{ ~~for $\alpha_i,\beta_i \in K_i$},
\end{equation*}
which makes it an elliptic curve as a polarized abelian variety (\cite[p.43--44]{Shimura-6}).
Let
\begin{equation*}
\theta_i = \left\{ \begin{array}{ll}
({-1+\sqrt{-d_i}})/{2}~ & \textrm{if $-d_i\equiv 1\pmod{4}$}\vspace{0.25cm}\\
\sqrt{-d_i}~ & \textrm{if $-d_i\equiv 2,3\pmod{4}$}
\end{array} \right.
\end{equation*}
and let $\Omega_i=\big[\theta_i~~1\big]\in M_{1\times 2}(\mathbb{C})$.
Then $\o_{K_i}=\mathbb{Z}[\theta_i]$ and $\Omega_i$ satisfies
\begin{eqnarray*}
L_i&=&\left\{\Omega_i\left[\begin{matrix} a\\b \end{matrix}\right]~\Big|~a,b\in\mathbb{Z} \right\},\\
E_i(\Omega_i \mathbf{x}, \Omega_i \mathbf{y})&=&{^t}\mathbf{x}J_1\mathbf{y} ~~~\textrm{for $\mathbf{x},\mathbf{y}\in\mathbb{R}^{2}$}.
\end{eqnarray*} 
On the other hand, we define a ring monomorphism $h_i:K_i\rightarrow M_{2}(\mathbb{Q})$ by
\begin{equation*}
\left[\begin{matrix}
\alpha_i\theta_i\\
\alpha_i
\end{matrix}\right]
=h_i(\alpha_i)\left[\begin{matrix}
\theta_i\\
1
\end{matrix}\right]~~\textrm{for $\alpha_i\in K_i$},
\end{equation*}
in other words, $h_i(\alpha_i)$ is the regular representation of $\alpha_i$ with respect to $\{\theta_i,1\}$.
Then $\theta_i$ becomes the CM-point of $\mathbb{H}_1$ induced from $h_i$ which corresponds to the 
elliptic curve $(\mathbb{C}/L_i, E_i)$ (\cite[p.684-685]{Shimura-4} or \cite[\S 24.10]{Shimura-6}).
\par
Now, let $Y=K_1\oplus K_2$ be a CM-algebra so that $[Y:\mathbb{Q}]=4$, $\mathcal{O}_Y=\mathcal{O}_{K_1}\oplus \mathcal{O}_{K_2}$ and $\rho=(\rho_1,\rho_2)\in Y$.
We denote by 
\begin{equation*}
v(\alpha)=\begin{bmatrix}\alpha_1\\ \alpha_2 \end{bmatrix}\quad\textrm{for $\alpha=(\alpha_1,\alpha_2)\in Y$}.
\end{equation*}
Let $L=\big\{v(\alpha)~\big|~\alpha\in\mathcal{O}_Y\big\}$ be a lattice in $\mathbb{C}^2$.
For $\mathbf{z}={^t}\begin{bmatrix}z_1& z_2\end{bmatrix}$ and $\mathbf{w}={^t}\begin{bmatrix}w_1& w_2\end{bmatrix}$ in $\mathbb{C}^2$, we define an $\mathbb{R}$-bilinear form $E(\mathbf{z},\mathbf{w})$ on $\mathbb{C}^2$ by
\begin{equation*}
E(\mathbf{z},\mathbf{w})=\sum_{i=1}^{2} E_i(z_i,w_i).
\end{equation*}
Then $E$ becomes a non-degenerate Riemann form on the complex torus $\mathbb{C}^2/L$ satisfying
\begin{equation*}
E\big(v(\alpha),v(\beta)\big)=\mathrm{Tr}_{Y/\mathbb{Q}}(\rho\alpha\overline{\beta})
=\sum_{i=1}^{2}\mathrm{Tr}_{K_i/\mathbb{Q}}(\rho_i\alpha_i\overline{\beta_i})
 \textrm{ ~~for $\alpha=(\alpha_i),\beta=(\beta_i) \in Y$},
\end{equation*}
which also makes it a polarized abelian variety (\cite[p.43--44, 129--130]{Shimura-6}).
Let 
\begin{equation*}
\Omega=
\begin{bmatrix}
\theta_1 &0 & 1&0 \\
0&\theta_2&0&1
\end{bmatrix}\in M_{2\times 4}(\mathbb{C}).
\end{equation*}
Then $\Omega$ satisfies
\begin{eqnarray*}
L&=&\left\{\Omega\left[\begin{matrix} \mathbf{a}\\\mathbf{b} \end{matrix}\right]~\Big|~\mathbf{a},\mathbf{b}\in\mathbb{Z}^2 \right\},\\
E(\Omega \mathbf{x}, \Omega \mathbf{y})&=&{^t}\mathbf{x}J_2\mathbf{y} ~~~\textrm{for $\mathbf{x},\mathbf{y}\in\mathbb{R}^{4}$}.
\end{eqnarray*} 
Here, we write $\Omega=\begin{bmatrix}\Omega_1&\Omega_2\end{bmatrix}=\begin{bmatrix}v(e_1)&v(e_2)&v(e_3)&v(e_{4})\end{bmatrix}$ with $\Omega_1,\Omega_2\in M_2(\mathbb{C})$ and $e_1, e_2, e_3, e_{4}\in Y$. 
Then $\{e_1,e_2,e_3,e_{4} \}$ is a free $\mathbb{Q}$-basis of $Y$, so we can define a ring monomorphism $h:Y\rightarrow M_{4}(\mathbb{Q})$ by
\begin{equation*}
\begin{bmatrix}
\alpha e_1\\
\vdots\\
\alpha e_{4}
\end{bmatrix}
=h(\alpha)\begin{bmatrix}
e_1\\
\vdots\\
e_{4}
\end{bmatrix}\quad\textrm{for $\alpha\in Y$},
\end{equation*}
that is, $h(\alpha)$ is the regular representation of $\alpha$ with respect to $\{e_1,e_2,e_3,e_{4}\}$.
One can then readily show that $Z_0=\Omega_2^{-1}\Omega_1=
\begin{bmatrix}
\theta_1  & 0\\
0&\theta_2
\end{bmatrix}
$ is the CM-point of $\mathbb{H}_2$ induced from $h$ which corresponds to the polarized abelian variety $(\mathbb{C}^2/L, E)$
(\cite[p.684-685]{Shimura-4} or \cite[\S 24.10]{Shimura-6}).
\par
Let $K=\mathbb{Q}(\sqrt{-d_1},\sqrt{-d_2})$ be an imaginary biquadratic field which is the composite field of $K_1$ and $K_2$.
Let $Y_{\mathbb{A}}=\prod_{i=1}^2(K_i)_{\mathbb{A}}$ and $Y_{\mathbb{A}}^\times=\prod_{i=1}^2(K_i)_{\mathbb{A}}^\times$.
We define a map $\varphi:K_\mathbb{A}^\times\rightarrow Y_\mathbb{A}^\times$ by
\begin{equation*}
\varphi(x)=\big(N_{K/K_i}(x) \big)_{1\leq i\leq 2} ~~\textrm{for $x\in K_\mathbb{A}^\times$}.
\end{equation*}
Then the map $h$ can be naturally extended to a homomorphism $Y_{\mathbb{A}}\rightarrow M_{4}(\mathbb{Q}_\mathbb{A})$, which we also denote by $h$.
Hence, for every $b\in K_{\mathbb{A}}^\times$ we get $\nu\big(h(\V(b))\big)=N_{K/\mathbb{Q}}(b)$ and $h\big(\V(b)^{-1}\big)\in G_{\mathbb{A}+}$ (\cite[p.172]{Shimura-6}).

\begin{proposition}[Shimura's reciprocity law]\label{reciprocity}
Let $Y$, $h$, $Z_0$ and $K$ be as above. Then for every $f\in\mathcal{A}_0^2(\mathbb{Q}_\mathrm{ab})$ which is finite at $Z_0$, the value $f(Z_0)$ belongs to $K_\mathrm{ab}$. Moreover, if $b\in K_\mathbb{A}^\times$, then $f^{h(\V(b)^{-1})}$ is finite at $Z_0$ and
\begin{equation*}
f(Z_0)^{[b,K]}=f^{h(\V(b)^{-1})}(Z_0).
\end{equation*}
\end{proposition}
\begin{proof}
\cite[Theorem 26.8]{Shimura-6}.
\end{proof}

\begin{remark}\label{class field}
For any $f\in\mathcal{A}_0^2(\mathbb{Q}_\mathrm{ab})$ that is finite at $Z_0$, the value $f(Z_0)$ indeed belongs to the class field $\widetilde{K_\mathrm{ab}}$ of $K$ corresponding to the kernel of $\varphi$.

\end{remark}

Now, we define a subset $\mathbb{H}_2^{\mathrm{diag}}$ of the Siegel upper half-space $\mathbb{H}_2$ by
\begin{equation*}
\mathbb{H}_2^{\mathrm{diag}}=\left\{ \begin{bmatrix}  z_1&0\\ 0&z_2
\end{bmatrix}
 ~|~ z_1,  z_2\in\mathbb{H}_1  \right\}.
\end{equation*}
Clearly, the CM-point $Z_0$ belongs to $\mathbb{H}_2^{\mathrm{diag}}$.
A function $f$ in $\mathcal{F}_N^1$ is called a \textit{modular unit of level $N$} if it has no zeros and poles on $\mathbb{H}_1$ \cite[p.36]{Kubert}.

\begin{proposition}\label{modular unit}
Let $N\geq 2$ be an integer and $f\in\mathcal{F}_N^2$.
Then $f$ has no zeros and poles on $\mathbb{H}_2^{\mathrm{diag}}$ if and only if there exist modular units $f_1, f_2$ of level $N$  such that
\begin{equation*}
f\left( \begin{bmatrix}  z_1&0\\ 0&z_2
\end{bmatrix}\right)= f_1(z_1)f_2(z_2).
\end{equation*}
\end{proposition}
\begin{proof}
See \cite[Theorem 4.2]{Eum}.
\end{proof}

Thus, we see from Proposition \ref{modular unit} that it is natural to investigate the class field of $K$ generated by the singular values $f(\theta_i)$ of modular units $f$.

\section{Siegel-Ramachandra invariants}\label{Siegel functions}
In this section we shall introduce the well-known modular units, Siegel functions, and review some necessary facts about the Siegel-Ramachandra invariants as singular values of Siegel functions.
\par
For a rational vector $\mathbf{r}=\left[\begin{matrix}r_1\\r_2\end{matrix}\right]\in({1}/{N})\mathbb{Z}^2\setminus\mathbb{Z}^2$ with an integer $N\geq 2$, we define the \textit{Siegel function} $g_{\mathbf{r}}(\tau)$ on $\tau\in\mathbb{H}_1$ by the following infinite product
\begin{equation*}
g_{\mathbf{r}}(\tau)=-q^{\frac{1}{2}\mathbf{B}_2(r_1)}e^{\pi i r_2(r_1-1)}(1-q^{r_1}e^{2\pi i r_2})\prod_{n=1}^\infty(1-q^{n+r_1}e^{2\pi i r_2})(1-q^{n-r_1}e^{-2\pi i r_2}),
\end{equation*}
where $\mathbf{B}_2(X)=X^2-X+{1}/{6}$ is the second Bernoulli polynomial and $q=e^{2\pi i\tau}$.
Then a Siegel function is a modular unit, namely, it is a modular function which has no zeros and poles on $\mathbb{H}_1$ (\cite{Siegel} or \cite[p.36]{Kubert}). 
Furthermore, the function $g_{\mathbf{r}}(\tau)^{12N}$ belongs to $\mathcal{F}_N^1$.

Let $F$ be a number field and $\mathfrak{m}$ be a modulus of $F$.
A modulus $\mathfrak{m}$ may be written as $\mathfrak{m}_0\mathfrak{m}_\infty$, where $\mathfrak{m}_0$ is an integral ideal of $F$ and $\mathfrak{m}_\infty$ is a product of distinct real infinite primes of $F$.
Let $I_F(\mathfrak{m})$ be the group of all fractional ideals of $F$ which are relatively prime to $\mathfrak{m}$ and $P_{F,1}(\mathfrak{m})$ be the subgroup of $I_F(\mathfrak{m})$ generated by the principal ideals $\alpha\o_F$, where $\alpha\in\o_F$ satisfies
\begin{itemize}
\item[\textup{(i)}] $\alpha\equiv 1\pmod{\mathfrak{m}_0}$
\item[\textup{(ii)}] $\tau(\alpha)>0$ for every real infinite prime $\tau$ dividing $\mathfrak{m}_\infty$.
\end{itemize}
Further, we let $\mathrm{Cl}_F(\mathfrak{m})=I_F(\mathfrak{m})/P_{F,1}(\mathfrak{m})$ be the ray class group of $F$ modulo $\mathfrak{m}$.
Then there exists a unique abelian extension $F_\mathfrak{m}$ of $F$ whose Galois group is isomorphic to $\mathrm{Cl}_F(\mathfrak{m})$ via the Artin map (\cite[Theorem 8.6]{Cox}).
Here, the field $F_\mathfrak{m}$ is called the \textit{ray class field} of $F$ modulo $\mathfrak{m}$.
And, it is well-known (\cite[Theorem 8.2]{Cox}) that any finite abelian extension of $F$ is contained in some ray class field $F_\mathfrak{m}$. 
\par
From now on, let $F=\mathbb{Q}(\sqrt{-d})$ be an imaginary quadratic field with square-free positive integer $d$ and let
\begin{equation*}
\theta=\left\{
\begin{array}{ll}
({-1+\sqrt{-d}})/{2}~ & \textrm{if $-d\equiv 1\pmod{4}$}\vspace{0.2cm}\\
\sqrt{-d}~ & \textrm{if $-d\equiv 2,3\pmod{4}$}
\end{array}
\right.
\end{equation*}
so that $\o_F=\mathbb{Z}[\theta]$.
Let $\mathfrak{f}$ be  a nontrivial proper integral ideal of $F$ and $N$ be the smallest positive integer in $\mathfrak{f}$.
For $C\in\mathrm{Cl}_F(\F)$, we take any integral ideal $\mathfrak{c}$ in $C$ and choose a basis $[\omega_1,\omega_2]$ of $\mathfrak{f}\mathfrak{c}^{-1}$ such that ${\omega_1}/{\omega_2}\in\mathbb{H}_1$.
Then we can write
\begin{equation*}
N=r_1\omega_1+r_2\omega_2
\end{equation*}
for some $r_1,r_2\in\mathbb{Z}$.
We define the \textit{Siegel-Ramachandra invariant} of conductor $\F$ at $C$  by 
\begin{equation*}
g_\F(C)=g_{\left[\begin{smallmatrix}r_1/N\\r_2/N\end{smallmatrix}\right]}({\omega_1}/{\omega_2})^{12N}.
\end{equation*}
This value depends only on the class $C$ and $\F$, not on the choice of $\mathfrak{c}$.

\begin{proposition}\label{imaginary generator}
Let $\F=N\o_F$ with an integer $N\geq 2$ and $C_0$ be the unit class in $\mathrm{Cl}_F(\F)$.
Assume that $F\neq \mathbb{Q}(\sqrt{-1}),\mathbb{Q}(\sqrt{-3})$.
Then the value
\begin{equation*}
g_\F(C_0)=g_{\left[\begin{smallmatrix}0\\1/N\end{smallmatrix}\right]}(\theta)^{12N}.
\end{equation*}
is a real algebraic integer.
Moreover, for any positive integer $n$, we get
\begin{equation*}
F_{(N)}=F\big(g_\F(C_0)^{n}\big).
\end{equation*}
\end{proposition}
\begin{proof}
\cite[Theorem 3.5 and Remark 3.6]{Jung}
\end{proof}

\section{Class fields over imaginary biquadratic fields (I)}\label{(I)}
We shall consider an imaginary biquadratic field $K=\mathbb{Q}(\sqrt{-d_1},\sqrt{-d_2})$ where
$d_1$, $d_2$ are square-free positive integer such that $-d_1\equiv 1\pmod{4}$, $-d_2\equiv 2,3\pmod{4}$ and $\gcd(d_1,d_2)=1$.
Then we have two imaginary quadratic subfields $K_1=\mathbb{Q}(\sqrt{-d_1})$, $K_2=\mathbb{Q}(\sqrt{-d_2})$ and one real quadratic subfield $K_3=\mathbb{Q}(\sqrt{d_1d_2})$ in $K$.

\begin{lemma}\label{integer-form}
The ring of integers $\mathcal{O}_K$ of $K$ is $\mathbb{Z}\big[(-1+\sqrt{-d_1})/2,\sqrt{-d_2}\big]$.
Consequently,
\begin{equation*}
\mathcal{O}_K=\left\{\frac{1}{2}\left[a+b\sqrt{-d_1}+c\sqrt{-d_2}+d\sqrt{d_1d_2}\right] ~\Big|~ a,b,c,d\in\mathbb{Z},~ a\equiv b~(\bmod{~2}),~ c\equiv d~(\bmod{~2}) \right\}.
\end{equation*}
\begin{proof}
By \cite[Chapter I, Theorem 9.5]{Janusz} we have $\mathcal{O}_K=\mathcal{O}_{K_1}\mathcal{O}_{K_2}=\mathbb{Z}\Big[\frac{-1+\sqrt{-d_1}}{2},\sqrt{-d_2}\Big]$.
Hence we deduce
\begin{eqnarray*}
\mathcal{O}_K &=& \left\{A+B\Big(\frac{-1+\sqrt{-d_1}}{2}\Big)+C\sqrt{-d_2}+D\Big(\frac{-\sqrt{-d_2}-\sqrt{d_1d_2}}{2}  \Big)~\Big|~ A,B,C,D\in\mathbb{Z}  \right\}\\
&=&\left\{ \frac{1}{2}\Big[ (2A-B)+B\sqrt{-d_1}+(2C-D)\sqrt{-d_2}-D\sqrt{d_1d_2}\Big]~\Big|~  A,B,C,D\in\mathbb{Z}        \right\} \\
&=&\left\{ \frac{1}{2}\Big[ a+b\sqrt{-d_1}+c\sqrt{-d_2}+d\sqrt{d_1d_2}\Big]~\Big|~ a,b,c,d\in\mathbb{Z},~ a\equiv b~(\bmod{~2}),~ c\equiv d~(\bmod{~2}) \right\} .
\end{eqnarray*}
\end{proof}
\end{lemma}

Let $N$ be a positive integer and $p$ be an odd prime not dividing $N$.
For a non-negative integer $\mu$, we set
\begin{eqnarray*}
S_{N,p,\mu}&=&\{a\in K^\times~|~a\equiv 1\pmod{Np^\mu\mathcal{O}_K}, ~\textrm{$a$ is prime to $p\mathcal{O}_K$} \},\\
S^{(i)}_{N,p,\mu}&=&\{a\in K_i^\times~|~a\equiv 1\pmod{Np^\mu\mathcal{O}_{K_i}}, ~\textrm{$a$ is prime to $p\mathcal{O}_{K_i}$} \}\quad\textrm{for $i=1,2$}.
\end{eqnarray*}
Further, we let 
\begin{eqnarray*}
H_{N,p,\mu}&=&S_{N,p,\mu+1}(S_{N,p,\mu}\cap\mathcal{O}_K^\times),\\
H_{N,p,\mu}^{(i)}&=&S_{N,p,\mu+1}^{(i)}(S_{N,p,\mu}^{(i)}\cap\mathcal{O}_{K_i}^\times)\quad\textrm{for $i=1,2$}.
\end{eqnarray*}
Then we achieve the isomorphisms
\begin{equation}\label{Galois group}
\begin{array}{cllll}
\mathrm{Gal}(K_{(Np^{\mu+1})}/K_{(Np^{\mu})})&\cong&  S_{N,p,\mu} \mathcal{O}_K^\times/S_{N,p,\mu+1}\mathcal{O}_K^\times &\cong& S_{N,p,\mu}/H_{N,p,\mu}\vspace{0.2cm}\\
\mathrm{Gal}((K_i)_{(Np^{\mu+1})}/(K_i)_{(Np^{\mu})})&\cong&  S_{N,p,\mu}^{(i)} \mathcal{O}_{K_i}^\times/S_{N,p,\mu+1}^{(i)}\mathcal{O}_{K_i}^\times &\cong& S_{N,p,\mu}^{(i)}/H_{N,p,\mu}^{(i)}
\end{array}
\end{equation}
by class field theory (\cite[Chapter V $\S$6]{Janusz}).
Since $N_{K/K_i}(H_{N,p,\mu})\subset H_{N,p,\mu}^{(i)}$ for $i=1,2$, we can define a homomorphism
\begin{equation*}
\widetilde{\V_{N,p,\mu}^{1,2}}:S_{N,p,\mu}/H_{N,p,\mu}\rightarrow S_{N,p,\mu}^{(1)}/H_{N,p,\mu}^{(1)}\times S_{N,p,\mu}^{(2)}/H_{N,p,\mu}^{(2)},
\end{equation*}
by
\begin{eqnarray*}
\widetilde{\V_{N,p,\mu}^{1,2}}(aH_{N,p,\mu})=\left(N_{K/K_1}(a)H_{N,p,\mu}^{(1)},~ N_{K/K_2}(a)H_{N,p,\mu}^{(2)}\right).
\end{eqnarray*}
Let $\widetilde{K_{N,p,\mu}^{1,2}}$  be the class field of $K$ corresponding to $\ker(\widetilde{\V_{N,p,\mu}^{1,2}})$.
Note that 
\begin{equation*}
\widetilde{K_{N,p,\mu}^{1,2}}=K_{(Np^{\mu})}(K_{(Np^{\mu+1})}\cap \widetilde{K_\mathrm{ab}}), 
\end{equation*}
where $\widetilde{K_\mathrm{ab}}$ is the class field of $K$ described in Remark \ref{class field}.

\begin{lemma}\label{widetilde}
With the notations as above, for a non-negative integer $\mu$ we obtain
\begin{eqnarray*}
\widetilde{K_{N,p,\mu}^{1,2}}=K_{(Np^\mu)}(K_1)_{(Np^{\mu+1})}(K_2)_{(Np^{\mu+1})}.
\end{eqnarray*}
\end{lemma}

\begin{proof}
For each index $i=1,2$, let $\gamma_i$ be a primitive generator of $(K_i)_{(Np^{\mu+1})}$ over $(K_i)_{(Np^\mu)}$.
Since $(K_i)_{(Np^\mu)}\subset K_{(Np^\mu)}$ for $i=1,2$, we attain
\begin{equation*}
K_{(Np^\mu)}(K_1)_{(Np^{\mu+1})}(K_2)_{(Np^{\mu+1})}=K_{(Np^\mu)}(\gamma_1,\gamma_2).
\end{equation*}
If $a H_{N,p,\mu}\in\ker(\widetilde{\V_{N,p,\mu}^{1,2}})$, then $N_{K/K_i}(a)\in H_{N,p,\mu}^{(i)}$ for $i=1,2$.
Since $\gamma_i\in(K_i)_{(Np^{\mu+1})}$, we have
\begin{equation*}
\gamma_i^{\big(\frac{K_{(Np^{\mu+1})}/K}{(a)}\big)}=\gamma_i^{\big(\frac{{(K_i)}_{(Np^{\mu+1})}/K_i}{N_{K/K_i}((a))}\big)}
=\gamma_i~~~\textrm{for $i=1,2$}
\end{equation*}
(\cite[Chapter III $\S$3]{Janusz}).
Hence $\widetilde{K_{N,p,\mu}^{1,2}}\supset K_{(Np^\mu)}(\gamma_1,\gamma_2)$.
\par
Conversely, let $bH_{N,p,\mu}\in S_{N,p,\mu}/H_{N,p,\mu}$ such that 
\begin{equation*}
\gamma_i^{\left(\frac{K_{(Np^{\mu+1})}/K}{(b)}\right)}=\gamma_i\quad \textrm{for $i=1,2$}.
\end{equation*}
Since $\gamma_i$ is the primitive generator of $(K_i)_{(Np^{\mu+1})}/(K_i)_{(Np^\mu)}$, we obtain $N_{K/K_i}(b)\in H_{N,p,\mu}^{(i)}$ for $i=1,2$ by (\ref{Galois group}).
Thus $b H_{N,p,\mu}\in\ker(\widetilde{\V_{N,p,\mu}^{1,2}})$, and so $\widetilde{K_{N,p,\mu}^{1,2}}\subset K_{(Np^\mu)}(\gamma_1,\gamma_2)$.
This completes the proof.
\end{proof}

For a non-negative integer $\mu$ and $i=1,2$, we let $\F_{\mu,i}=Np^\mu\o_{K_i}$ and $C_{\mu,i}$ be the unit class in 
$\mathrm{Cl}_{K_i}(\F_{\mu,i})$.

\begin{corollary}\label{primitive generator}
Assume that $K_1,K_2\neq \mathbb{Q}(\sqrt{-1}),\mathbb{Q}(\sqrt{-3})$.
For any positive integers $n_1$, $n_2$, the value
\begin{equation*}
\prod_{i=1}^2 g_{\F_{\mu+1,i}}(C_{\mu+1,i})^{n_i}
\end{equation*}
generates $\widetilde{K_{N,p,\mu}^{1,2}}$ over $K_{(Np^\mu)}$.
\end{corollary}

\begin{proof}
It follows from Proposition \ref{imaginary generator} and Lemma \ref{widetilde} that
\begin{equation*}
\widetilde{K_{N,p,\mu}^{1,2}}=K_{(Np^\mu)}\Big(g_{\F_{\mu+1,1}}(C_{\mu+1,1})^{n_1},~g_{\F_{\mu+1,2}}(C_{\mu+1,2})^{n_2}\Big).
\end{equation*}
Note that 
\begin{equation*}
\big|(g_{\F_{\mu+1,i}}(C_{\mu+1,i})^{n_i})^\tau\big|>\big|g_{\F_{\mu+1,i}}(C_{\mu+1,i})^{n_i}\big|
\end{equation*}
for $i=1,2$ and $\tau\in\mathrm{Gal}((K_i)_{(Np^{\mu+1})}/K_i)\setminus\{\mathrm{Id}\}$
(\cite[Theorem 3.5 and Remark 3.6]{Jung}).
Hence 
\begin{equation*}
\left(\prod_{i=1}^2 g_{\F_{\mu+1,i}}(C_{\mu+1,i})^{n_i}\right)^\tau\neq\prod_{i=1}^2 g_{\F_{\mu+1,i}}(C_{\mu+1,i})^{n_i}
\end{equation*}
for $\tau\in\mathrm{Gal}\big(\widetilde{K_{N,p,\mu}^{1,2}}/K_{(Np^\mu)}\big)\setminus\{\mathrm{Id}\}$.
This proves the corollary.\end{proof}

In this section we shall consider only the case $\mu=0$.
As for the other cases, see Section \ref{(II)}.

\begin{lemma}\label{S-order}
With the notation as above, we get
\begin{eqnarray*}                                                                                                                                                                                                                                                                                                                                                                                                                                                                                                                                                                                                                                                                                                                                                                                                                                                                                                                                                                                                                                                                                                                                                                                                                                                                                                                                                                                                                                                                                                                                                                               
\big|S_{N,p,0}/S_{N,p,1}\big|=(p-1)\prod_{j=1}^3 m_{p,j},
\end{eqnarray*}
where $m_{p,i}=p-\big(\frac{-d_i}{p}\big)$ for $i=1,2$ and $m_{p,3}=p-\big(\frac{d_1 d_2}{p}\big)$.
\end{lemma}
\begin{proof}
Note that for any coset $aS_{N,p,1}$ in $S_{N,p,0}/S_{N,p,1}$, there is an element $a'\in  S_{N,p,0}\cap\mathcal{O}_K$ 
such that $aS_{N,p,1}=a'S_{N,p,1}$.
Here we claim that the map
\begin{equation}\label{group isomorphism}
\begin{array}{ccl}
\psi:S_{N,p,0}/S_{N,p,1}&\longrightarrow& (\mathcal{O}_K/p\mathcal{O}_K)^\times   \\
\omega S_{N,p,1} &\longmapsto& \omega+p\mathcal{O}_K \textrm{~~~ for $\O\in S_{N,p,0}\cap\mathcal{O}_K$}
\end{array}
\end{equation}
is an isomorphism.
Indeed, if $\O, \O'\in S_{N,p,0}\cap\mathcal{O}_K$ such that $\O S_{N,p,1}=\O'S_{N,p,1}$, then $\O'=\O\alpha$ for some $\alpha\in S_{N,p,1}$.
We can then write $\alpha=a/b$ with $a,b\in \mathcal{O}_K$ such that $a,b$ are prime to $Np\o_K$ and $a\equiv b\pmod{Np\o_K}$.
So $\O' b= \O a\equiv \O b\pmod{p\o_K}$, from which we have $\O'\equiv \O \pmod{p\o_K}$.
Thus $\psi$ is a well-defined homomorphism.
If $\omega S_{N,p,1}\in\ker(\psi)$ then $\O\equiv 1\pmod{p\o_K}$.
Since $\gcd(N,p)=1$, we obtain $\O\equiv 1\pmod{Np\o_K}$, and so $\psi$ is injective.
For given $\omega+p\mathcal{O}_K\in(\o_K/p\o_K)^\times$, we can find $a\in\o_K$ such that $\O+pa\equiv 1\pmod{N\o_K}$.
Therefore $\psi$ is surjective, and hence the claim is proved.
By \cite[Proposition 4.2.12]{Cohen0} we derive
\begin{equation*}
|(\o_K/p\o_K)^\times|=\mathcal{N}(p\o_K)\cdot\displaystyle\prod_{\substack{\mathfrak{p}\,|\, p\mathcal{O}_K\\ \mathfrak{p}~\textrm{prime} }}\left(1-\frac{1}{\mathcal{N}(\mathfrak{p})}\right).
\end{equation*}
By considering all possible prime ideal factorization of $p\o_K$, we get the conclusion
(\cite[p.74 and p.116]{Marcus}).
\end{proof}
From now on, we assume that $N\neq 2$ and $K_1,K_2\neq \mathbb{Q}(\sqrt{-1}),\mathbb{Q}(\sqrt{-3})$.
Let $\varepsilon_0$ be the fundamental unit of the real quadratic field $K_3$.
Since $d_1\equiv 3\pmod{4}$, the norm of $\varepsilon_0$ is 1.
We let $m_0$ be the smallest positive integer such that 
$\varepsilon_0^{m_0}\equiv \pm1\pmod{N\o_K}$, and set 
\begin{equation*}
\varepsilon_0' =\left\{
\begin{array}{ll}
\varepsilon_0^{m_0} &\textrm{if $N\neq 1$ and $\varepsilon_0^{m_0}\equiv +1\pmod{N\o_K}$}\\
-\varepsilon_0^{m_0} &\textrm{if $N\neq 1$ and $\varepsilon_0^{m_0}\equiv -1\pmod{N\o_K}$}\\
\varepsilon_0 &\textrm{if $N=1$}.
\end{array}
\right.
\end{equation*}
Further, we let $n_0$ be the smallest positive integer for which $(\varepsilon_0')^{n_0}\equiv 1\pmod{Np\o_K}$.

\begin{lemma}\label{H-order}
With the assumptions as above, we deduce
\begin{eqnarray*}
\big|H_{N,p,0}/S_{N,p,1}\big|=\left\{
\begin{array}{ll}
n_0& \textrm{if $N\neq 1$}\\
n_0\cdot Q(K)& \textrm{if $N=1$ and $n_0$ is even}\\
2n_0\cdot Q(K)& \textrm{if $N=1$ and $n_0$ is odd},
\end{array}\right.
\end{eqnarray*}
where $Q(K)=\big[\o_K^\times:\o_{K_1}^\times\o_{K_2}^\times\o_{K_3}^\times\big]$.
\end{lemma}
\begin{proof}
By the assumption, $\o_{K_1}^\times=\o_{K_2}^\times=\{\pm 1\}$ so that $\o_{K_1}^\times\o_{K_2}^\times\o_{K_3}^\times=\o_{K_3}^\times$.
Observe that
\begin{equation*}
\big|H_{N,p,0}/S_{N,p,1}\big|=\big|H_{N,p,0}/S_{N,p,1}(S_{N,p,0}\cap \o_{K_3}^\times)\big|\cdot \big|S_{N,p,1}(S_{N,p,0}\cap \o_{K_3}^\times)/S_{N,p,1}\big|.
\end{equation*}
\par
First, we consider the group $H_{N,p,0}/S_{N,p,1}(S_{N,p,0}\cap\o_{K_3}^\times)$.
Suppose $N\neq 1$.
Then we claim that $S_{N,p,0}\cap\o_{K}^\times\subset\mathbb{R}$, namely, $H_{N,p,0}=S_{N,p,1}(S_{N,p,0}\cap \o_{K_3}^\times)$.
Indeed, let $\varepsilon\in S_{N,p,0}\cap \o_K^\times$.
By Lemma \ref{integer-form} we can write 
\begin{equation*}
\varepsilon=\frac{1}{2}\Big[a+b\sqrt{-d_1}+c\sqrt{-d_2}+d\sqrt{d_1d_2}\Big]
\end{equation*} 
with $a,b,c,d\in\mathbb{Z}$ such that $a\equiv b\pmod{2}$, $c\equiv d\pmod{2}$.
Since $\varepsilon\equiv 1\pmod{N\o_K}$, we have $a\equiv 2\pmod{N}$ and so $a\neq 0$.
Here we note that
\begin{eqnarray*}
\varepsilon^2&=&\frac{1}{4}\Big[a^2-b^2d_1-c^2d_2+d^2d_1d_2+2\big\{(ab+cdd_2)\sqrt{-d_1}+(ac+bdd_1)\sqrt{-d_2}+(ad-bc)\sqrt{d_1d_2}   \big\}\Big]\\
&\in&\o_{K_1}^\times\o_{K_2}^\times\o_{K_3}^\times\subset\mathbb{R}
\end{eqnarray*}
because $Q(K)=1~\textrm{or}~2$.
And, we obtain
\begin{equation}\label{real}
\begin{array}{lll}
ab+cdd_2&=&0\\
ac+bdd_1&=&0, 
\end{array}
\end{equation}
hence $d(b^2d_1-c^2d_2)=0$.
If $d=0$, then $ab=ac=0$ which enables us to get $b=c=0$.
Thus $\varepsilon=a/2\in\mathbb{R}$ in this case.
Now, suppose that $d\neq 0$ and $b^2d_1-c^2d_2=0$.
Then it follows from (\ref{real}) that $-{(a^2b)}/{(dd_2)}=ac=-bdd_1$, and we have $b(a^2-d^2d_1d_2)=0$.
If $b=0$ then $c=0$, and hence $\varepsilon=(a+d\sqrt{d_1d_2})/2\in\mathbb{R}$.
If $a^2-d^2d_1d_2=0$, then we achieve
\begin{eqnarray*}
N_{K/K_1}(\varepsilon)&=&\frac{1}{4}\big[a^2-b^2d_1+c^2d_2-d^2d_1d_2+2(ab-cdd_2)\sqrt{-d_1}\big]\\
&=&\frac{1}{2}(ab-cdd_2)\sqrt{-d_1}\\
&\neq&\pm 1.
\end{eqnarray*}
This contradicts the assumption $\o_{K_1}=\{\pm 1\}$, and the claim is justified.
\par
If $N=1$, then $S_{1,p,0}\cap\o_{K}^\times=\o_{K}^\times$ and $S_{1,p,0}\cap\o_{K_3}^\times=\o_{K_3}^\times$.
Hence one can show $S_{1,p,1} \cap \o_K^\times \subset \mathbb{R}$ in a similar fashion as in the above claim.
And, we establish
\begin{equation*}
\begin{array}{lll}
H_{1,p,0}/S_{1,p,1}(S_{1,p,0}\cap \o_{K_3}^\times)&=&S_{1,p,1}\o_{K}^\times /S_{1,p,1}\o_{K_3}^\times\vspace{0.1cm}\\
&\cong& \o_{K}^\times/\o_{K_3}^\times(S_{1,p,1}\cap \o_{K}^\times)\vspace{0.1cm}\\
&\cong&\o_{K}^\times/\o_{K_3}^\times.
\end{array}
\end{equation*}
Therefore we attain
\begin{equation}\label{first-order}
\big|H_{N,p,0}/S_{N,p,1}(S_{N,p,0}\cap \o_{K_3}^\times)\big|=\left\{
\begin{array}{ll}
1&\textrm{if $N\neq 1$}\\
Q(K)&\textrm{if $N=1$}.
\end{array}\right.
\end{equation}
\par
Next, we consider the group $S_{N,p,1}(S_{N,p,0}\cap \o_{K_3}^\times)/S_{N,p,1}$.
If $N\neq 1$, then $S_{N,p,0}\cap \o_{K_3}^\times=\{(\varepsilon_0')^n~|~n\in\mathbb{Z} \}$, and hence we have
\begin{equation*}
\big|S_{N,p,1}(S_{N,p,0}\cap \o_{K_3}^\times)/S_{N,p,1}\big|=\big|\langle\varepsilon_0'S_{N,p,1} \rangle\big|=n_0.
\end{equation*}
Now, suppose $N=1$.
If $n_0$ is even, then $(\varepsilon_0')^{{n_0}/{2}}$ is the root of the equation $X^2\equiv 1\pmod{p\o_{K_3}}$.
Write $(\varepsilon_0')^{{n_0}/{2}}=\alpha+\beta\sqrt{d_1 d_2}$ with $\alpha, \beta\in\mathbb{Z}$.
Then we get
\begin{eqnarray*}
\alpha^2+d_1 d_2 \beta^2&\equiv& 1 \pmod{p}\\
2\alpha\beta&\equiv& 0 \pmod{p}.
\end{eqnarray*}
If $p$ divides $\alpha$, then $d_1d_2\beta^2 \equiv 1 \pmod{p}$.
Since $N_{K_3/\mathbb{Q}}\big((\varepsilon_0')^{{n_0}/{2}}\big)=\alpha^2- d_1 d_2\beta^2=1$, it is a contradiction.
Thus $p$ divides $\beta$, and we achieve $(\varepsilon_0')^{{n_0}/{2}}\equiv -1\pmod{p\o_K}$ by the minimality of $n_0$.
If $n_0$ is odd, then $(\varepsilon_0')^n\not\equiv -1\pmod{p\o_K}$ for all $n\in\mathbb{Z}_{>0}$.
Indeed, if $(\varepsilon_0')^n\equiv -1\pmod{p\o_K}$ for some $n\in\mathbb{Z}_{>0}$, then 
$(\varepsilon_0')^{2n}\equiv 1\pmod{p\o_K}$.
By definition of $n_0$ we see that $n_0$ divides $2n$.
Since $n_0$ is odd, $n_0$ divides $n$.
But $(\varepsilon_0')^n\not\equiv 1\pmod{p\o_K}$, so it gives a contradiction.
Since $ \o_{K_3}^\times=\{\pm(\varepsilon_0')^n~|~n\in\mathbb{Z}\}$, we derive that
\begin{equation*}
S_{1,p,1}(S_{1,p,0}\cap \o_{K_3}^\times)/S_{1,p,1} = S_{1,p,1}\o_{K_3}^\times/S_{1,p,1}\cong \left\{
\begin{array}{ll}
\mathbb{Z}/{n_0\mathbb{Z}} & \textrm{if $n_0$ is even}\\
\mathbb{Z}/2\mathbb{Z}\times\mathbb{Z}/{n_0\mathbb{Z}} & \textrm{if $n_0$ is odd}.
\end{array}\right.
\end{equation*}
And, we deduce
\begin{equation}\label{second-order}
\big|S_{N,p,1}(S_{N,p,0}\cap \o_{K_3}^\times)/S_{N,p,1}\big| = \left\{
\begin{array}{ll}
n_0 & \textrm{if $N\neq 1$ or $n_0$ is even}\\
2n_0 & \textrm{if $N=1$ and $n_0$ is odd}.
\end{array}\right.
\end{equation}
Therefore, the lemma follows from (\ref{first-order}) and (\ref{second-order}).
\end{proof}

\begin{lemma}\label{unit norm}
Let $F=\mathbb{Q}(\sqrt{d})$ be a quadratic field with square-free integer $d$ and $p$ be an odd prime such that $p\nmid d$.
Then we attain
\begin{equation*}
\{\O+p\o_F\in (\o_F/p\o_F)^\times~|~N_{F/\mathbb{Q}}(\O)\equiv 1\pmod{p}\}\cong\mathbb{Z}/m\mathbb{Z},
\end{equation*}
with $m=p-\big(\frac{d}{p}\big)$.
\end{lemma}
\begin{proof}
First, consider an affine curve 
\begin{equation*}
\mathcal{C}:x^2-d y^2=1
\end{equation*}
defined over a finite field $\mathbb{F}_p$.
Then $\mathcal{C}$ is smooth since $p\nmid d$.
We define a group law on $\mathcal{C}$ by
\begin{equation*}
(r,s)\oplus(t,u)=(rt+d su, ru+st)
\end{equation*}
for points $(r,s)$, $(t,u)$ on $\mathcal{C}$.
Then the group $\mathcal{C}(\mathbb{F}_p)$ is isomorphic to $\mathbb{Z}/m\mathbb{Z}$ (\cite{Lemmermeyer}).
\par
On the other hand, since $\gcd(2,p)=1$, we obtain
\begin{equation*}
\o_F/p\o_F=\{x+y\sqrt{d}+p\o_F~|~x,y\in\mathbb{F}_p \}.
\end{equation*}
And, the map
\begin{equation*}
\begin{array}{ccc}
\{\O+p\o_F\in (\o_F/p\o_F)^\times~|~N_{F/\mathbb{Q}}(\O)\equiv 1\pmod{p}\}&\longrightarrow&\mathcal{C}(\mathbb{F}_p)\\
x+y\sqrt{d}+p\o_F&\longmapsto&(x,y)
\end{array}
\end{equation*}
is an isomorphism, which completes the proof.
\end{proof}

Now, we further assume that $p\nmid d_1 d_2$.
Then we derive
\begin{equation*}
p\o_K=\left\{
\begin{array}{ll}
\mathfrak{p}_1\mathfrak{p}_2\mathfrak{p}_3\mathfrak{p}_4&\textrm{if $\big(\frac{-d_1}{p}\big)=\big(\frac{-d_2}{p}\big)=1$}\\
\mathfrak{p}_1\mathfrak{p}_2&\textrm{otherwise}
\end{array}\right.
\end{equation*}
where the $\mathfrak{p}_i$ are distinct prime ideals of $K$ (\cite[p.116]{Marcus}).
By making use of (\ref{Galois group}), Lemma \ref{S-order} and \ref{H-order}, one can find the extension degree of $K_{(Np)}$ 
 over $K_{(N)}$.
Here we note that 
\begin{equation*}
\varepsilon_0'+p\o_{K_3}\in\{\O+p\o_{K_3}\in(\o_{K_3}/p\o_{K_3})^\times~|~N_{K_3/\mathbb{Q}}(\O)\equiv 1\pmod{p}\}
\end{equation*}
and so $n_0$ divides $m_{p,3}$ by Lemma \ref{unit norm}.

For $i=1,2$, let
\begin{eqnarray*}
W_{N,p,0}^i&=&\left\{
\begin{array}{ll}
\big\{\O\in S_{N,p,0}~|~ N_{K/K_i}(\O)\equiv  1\pmod{Np\o_{K_i}}  \big\}&\textrm{if $N\neq 1$}\\ 
\big\{\O\in S_{1,p,0}~|~ N_{K/K_i}(\O)\equiv \pm 1\pmod{p\o_{K_i}} \big\}&\textrm{if $N= 1$},
\end{array}\right.\\
\end{eqnarray*}
and let
\begin{eqnarray*}
W_{N,p,0}^{1,2}&=&W_{N,p,0}^1 \cap W_{N,p,0}^2,
\end{eqnarray*}
so that $\ker(\widetilde{\V_{N,p,0}^{1,2}})=W_{N,p,0}^{1,2}/H_{N,p,0}$.
Here, we mean by $\widetilde{W_{N,p,0}^{1,2}}$  the image of $W_{N,p,0}^{1,2}/S_{N,p,1}$ in $(\o_K/p\o_K)^\times$ via the isomorphism (\ref{group isomorphism}).

\begin{theorem}\label{K_{1,2}}
Let $N\neq 2$ be a positive integer and $p$ be an odd prime not dividing $Nd_1d_2$.
Assume that $K_1,K_2\neq \mathbb{Q}(\sqrt{-1}),\mathbb{Q}(\sqrt{-3})$.
Then we deduce
\begin{equation*}
\big[K_{(Np)}:\widetilde{K_{N,p,0}^{1,2}}\big] =\left\{
\begin{array}{ll}
\displaystyle\frac{m_{p,3}}{n_0}& \textrm{if $N\neq 1$}\vspace{0.2cm}\\
\displaystyle\frac{4m_{p,3}}{n_0\cdot Q(K)}& \textrm{if $N=1$ and $n_0$ is even}\vspace{0.2cm}\\
\displaystyle\frac{2m_{p,3}}{n_0\cdot Q(K)}& \textrm{if $N=1$ and $n_0$ is odd},
\end{array}\right.
\end{equation*}
Hence, we get
\begin{equation*}
K_{(Np)}=\widetilde{K_{N,p,0}^{1,2}}\quad\textrm{if and only if}\quad \textrm{$N\neq 1$ and $n_0=m_{p,3}$}.
\end{equation*}
\end{theorem}

\begin{proof}
Let
\begin{eqnarray*}
\widetilde{W_{N,p,0}^{1,2}~'}=
\{\O+p\o_K\in (\o_K/p\o_K)^\times~|~ N_{K/K_i}(\O)\equiv  1\pmod{p\o_{K_i}}~~\textrm{for $i=1,2$}\},
\end{eqnarray*}
and let $d_i^{-1}$ be the inverse of $d_i$ in $(\mathbb{Z}/p\mathbb{Z})^\times$ for $i=1,2$.
For $\O+p\o_K\in\o_K/p\o_K$ we can write $\O=a+b\sqrt{-d_1}+c\sqrt{-d_2}+d\sqrt{d_1d_2}$ 
with $a,b,c,d\in\mathbb{Z}$ due to the fact $\gcd(2,p)=1$.
\par
First, we consider the group $\widetilde{W_{N,p,0}^{1,2}}$.
If $\O+p\o_K\in \widetilde{W_{N,p,0}^{1,2}~'}$, then we attain
\begin{equation}\label{W_1}
\begin{array}{rll}
a^2-b^2d_1+c^2d_2-d^2d_1d_2&\equiv& 1\pmod{p}\\
ab-cdd_2&\equiv& 0\pmod{p},
\end{array}
\end{equation}
and
\begin{equation}\label{W_2}
\begin{array}{rll}
a^2+b^2d_1-c^2d_2-d^2d_1d_2&\equiv&  1\pmod{p}\\
ac-bdd_1&\equiv& 0\pmod{p}.
\end{array}
\end{equation}
Since $ c^2dd_2\equiv abc\equiv b^2dd_1\pmod{p}$, we obtain $d(b^2d_1-c^2d_2)\equiv 0\pmod{p}$.
If $b^2d_1-c^2d_2\equiv 0\pmod{p}$, then $p$ must divide $b,c$ and hence we get $a^2-d^2d_1d_2\equiv 1\pmod{p}$.
Indeed, if $\big(\frac{d_1d_2}{p} \big)=-1$, it is clear.
On the contrary, suppose that $\big(\frac{d_1d_2}{p} \big)=1$ and $b,c\not\equiv 0\pmod{p}$.
Then $d_1^{-1}d_2\equiv D^2 \pmod{p}$ for some $D\in\mathbb{Z}$, which yields $b\equiv \pm cD \pmod{p}$ and 
$a\equiv \pm dDd_1 \pmod{p}$.
Thus we deduce that
\begin{equation*}
a^2-b^2d_1+c^2d_2-d^2d_1d_2 \equiv d^2d_1d_2-c^2d_2+c^2d_2-d^2d_1d_2 \equiv 0\pmod{p},
\end{equation*}
which contradicts (\ref{W_1}).
If $b^2d_1-c^2d_2\not\equiv 0\pmod{p}$ and $d\equiv 0\pmod{p}$, then $ab\equiv ac\equiv 0\pmod{p}$ and so  $a\equiv 0\pmod{p}$.
But it follows from (\ref{W_1}), (\ref{W_2}) that 
$b^2d_1-c^2d_2\equiv  -1\equiv 1\pmod{p}$, which is a contradiction.
Therefore we derive
\begin{equation*}
\begin{array}{ccl}
\widetilde{W_{N,p,0}^{1,2}~'}&=&\big\{a+d\sqrt{d_1d_2}+p\o_K\in(\o_K/p\o_K)^\times ~|~ a^2-d^2d_1d_2\equiv 1~(\bmod{~p})\big\}\vspace{0.1cm}\\
&\cong&\big\{\O+p\o_{K_3}\in(\o_{K_3}/p\o_{K_3})^\times ~|~ N_{K_3/\mathbb{Q}}(\O)\equiv 1~(\bmod{~p})\big\}\vspace{0.1cm}\\
&\cong&\mathbb{Z}/m_{p,3}\mathbb{Z}\quad(\textrm{by Lemma \ref{unit norm}}).
\end{array}
\end{equation*}
If $N\neq 1$, then $\widetilde{W_{N,p,0}^{1,2}}=\widetilde{W_{N,p,0}^{1,2}~'}$ and by Lemma \ref{H-order} we attain
\begin{equation*}
\big[K_{(Np)}:\widetilde{K_{N,p,0}^{1,2}}\big]=|W_{N,p,0}^{1,2}/H_{N,p,0}|=\frac{\big|\widetilde{W_{N,p,0}^{1,2}}\big|}{|H_{N,p,0}/S_{N,p,1}|}=\frac{m_{p,3}}{n_0}.
\end{equation*}
For $i=1,2$, we let 
\begin{equation*}
\begin{array}{lll}
A_i&=&\big\{b\sqrt{-d_1}+c\sqrt{-d_2}+p\o_K\in(\o_K/p\o_K)^\times ~|~ b^2d_1-c^2d_2\equiv (-1)^i~(\bmod{~p})\big\}\vspace{0.1cm}\\
&=&\big\{b\sqrt{-d_1}+c\sqrt{-d_2}+p\o_K\in(\o_K/p\o_K)^\times ~|~ b^2-c^2d_1^{-1}d_2\equiv (-1)^i d_1^{-1}~(\bmod{~p})\big\}\vspace{0.1cm}\\
&\neq&\phi\quad(\textrm{by Lemma \ref{unit norm}}).
\end{array}
\end{equation*}
If $N=1$, we choose $a_1, a_2\in\o_K$ so that 
$a_i+p\o_K\in A_i$ for $i=1,2$.
Then it satisfies
\begin{equation*}
\begin{array}{ll}
N_{K/K_1}(a_{1})\equiv -1\pmod{p},& N_{K/K_2}(a_{1})\equiv 1\pmod{p},\\
N_{K/K_1}(a_{2})\equiv 1\pmod{p},& N_{K/K_2}(a_{2})\equiv -1\pmod{p}.
\end{array}
\end{equation*}
Hence we achieve
\begin{equation*}
\widetilde{W_{1,p,0}^{1,2}}=\bigsqcup_{0\leq i,j<2} a_1^i a_2^j\cdot\widetilde{W_{1,p,0}^{1,2}~'}.
\end{equation*}  
Thus $\big|\widetilde{W_{1,p,0}^{1,2}}\big|=4m_{p,3}$, and so by Lemma \ref{H-order} we derive
\begin{equation*}
\big[K_{(p)}:\widetilde{K_{1,p,0}^{1,2}}\big]=\big|W_{1,p,0}^{1,2}/H_{1,p,0}\big|=\frac{\big|\widetilde{W_{1,p,0}^{1,2}}\big|}{\big|H_{1,p,0}/S_{1,p,1}\big|}=\left\{
\begin{array}{ll}
\displaystyle\frac{4m_{p,3}}{n_0\cdot Q(K)}& \textrm{if $n_0$ is even}\vspace{0.2cm}\\
\displaystyle\frac{2m_{p,3}}{n_0\cdot Q(K)}& \textrm{if $n_0$ is odd}.
\end{array}
\right. 
\end{equation*}
This proves the theorem.
\end{proof}

\begin{example}\label{example1}
Let $K=\mathbb{Q}(\sqrt{-15},\sqrt{-26})$.
Then $K_1=\mathbb{Q}(\sqrt{-15})$, $K_2=\mathbb{Q}(\sqrt{-26})$, 
$K_3=\mathbb{Q}(\sqrt{390})$ and $\varepsilon_0=79+4\sqrt{390}$.
We set $N=5$ and $p=37$ so that $m_0=5$, $\varepsilon_0'=-\varepsilon_0^{m_0}$ and $n_0=38=m_{p,3}$.
Thus by Corollary \ref{primitive generator} and Theorem \ref{K_{1,2}}, for any positive integers $n_1$, $n_2$,
\begin{equation}\label{example 1-1}
K_{(185)}=\widetilde{K_{5,37,0}^{1,2}}=K_{(5)}
\Big(g_{\left[\begin{smallmatrix}0\\1/185\end{smallmatrix}\right]}(\theta_{1})^{2220n_1}
g_{\left[\begin{smallmatrix}0\\1/185\end{smallmatrix}\right]}(\theta_{2})^{2220n_2}
\Big)
\end{equation}
where $\theta_1=({-1+\sqrt{-15}})/{2}$ and $\theta_2=\sqrt{-26}$.

\end{example}

\section{Class fields over imaginary biquadratic fields (II)}\label{(II)}
Following the previous section we shall consider the more general case $\mu>0$.
\par
Let $N$ be a positive integer and $p$ be an odd prime not dividing $N$.
We use the same notations as in Section \ref{(I)}.

\begin{lemma}\label{S-order2}
For a positive integer $\mu$, we have
\begin{equation*}
\big|S_{N,p,\mu}/S_{N,p,\mu+1}\big|=p^4.
\end{equation*}
\end{lemma}
\begin{proof}
Now that
$S_{N,p,\mu}/S_{N,p,\mu+1}$ is isomorphic to $\mathcal{O}_K/p\mathcal{O}_K$ by a mapping 
\begin{eqnarray*}
S_{N,p,\mu}/S_{N,p,\mu+1} \longrightarrow \mathcal{O}_K/p\mathcal{O}_K \textrm{\qquad\qquad\quad ~~~~~}\\
(1+Np^\mu\O)S_{N,p,\mu+1} \longmapsto \O+p\mathcal{O}_K \textrm{~~~ for $\O\in\mathcal{O}_K$},
\end{eqnarray*}
we obtain $S_\mu/S_{\mu+1}\cong (\mathbb{Z}/p\mathbb{Z})^{4}$.
\end{proof}

Let $\varepsilon_0$ be the fundamental unit of the real quadratic field $K_3$, $\ell_0$ be the smallest positive integer such that 
$\varepsilon_0^{\ell_0}\equiv 1\pmod{Np\o_K}$ and $\mu_0$ be the maximal positive integer satisfying $\varepsilon_0^{\ell_0}\equiv 1\pmod{Np^{\mu_0}\o_K}$.
Write 
\begin{equation*}
\varepsilon_0^{\ell_0}=1+Np^{\mu_0}(\alpha_0+\beta_0\sqrt{d_1d_2})
\end{equation*}
with $\alpha_0,\beta_0\in\mathbb{Z}$.
By the maximality of $\mu_0$ we have $\alpha_0+\beta_0\sqrt{d_1d_2}\not\in p\o_{K_3}$.
Since
\begin{equation*}
1=N_{K_3/\mathbb{Q}}(\varepsilon_0^{\ell_0})\equiv 1+2Np^{\mu_0}\alpha_0 \pmod{Np^{\mu_0+1}},
\end{equation*}
we get $p~|~\alpha_0$ and $p\nmid\beta_0$.

\begin{lemma}\label{H-order2}
Assume that $K_1,K_2\neq \mathbb{Q}(\sqrt{-1}), \mathbb{Q}(\sqrt{-3})$.
Then for a positive integer $\mu$ we get
\begin{equation*}
\big|H_{N,p,\mu}/S_{N,p,\mu+1}\big|=\left\{
\begin{array}{ll}
1&\textrm{if $\mu<\mu_0$}\\
p&\textrm{if $\mu\geq \mu_0$}.
\end{array}\right.
\end{equation*}
\end{lemma}
\begin{proof}
In a similar way as in the proof of Lemma \ref{H-order}, one can verify that if $\ell_0$ is odd, then $\varepsilon_0^n\not\equiv -1\pmod{Np\o_K}$ for all $n\in\mathbb{Z}_{>0}$.
And, if $\ell_0$ is even, then we have either $\varepsilon_0^{{\ell_0}/{2}}\equiv -1\pmod{Np\o_K}$ or $\varepsilon_0^n\not\equiv -1\pmod{Np\o_K}$ for all $n\in\mathbb{Z}_{>0}$.
Indeed, suppose $\varepsilon_0^{{\ell_0}/{2}}\not\equiv -1\pmod{Np\o_K}$.
If $\varepsilon_0^{n}\equiv -1\pmod{Np\o_K}$ for some $n\in\mathbb{Z}_{>0}$, then $\ell_0~|~2n$ and so $({\ell_0}/{2})~|~n$.
It contradicts the fact $(\varepsilon_0^{{\ell_0}/{2}})^m\not\equiv -1\pmod{Np\o_K}$ for all $m\in\mathbb{Z}_{>0}$.
\par
Next, we claim that if $\ell_0$ is even and $\varepsilon_0^{{\ell_0}/{2}}\equiv -1\pmod{Np\o_K}$, then $\mu_0$ is the maximal positive integer such that $\varepsilon_0^{{\ell_0}/{2}}\equiv -1\pmod{Np^{\mu_0}\o_K}$.
For, write $\varepsilon_0^{{\ell_0}/{2}}=-1+Np\O$ for some $\O\in\o_{K_3}$. 
Then
\begin{equation*}
\varepsilon_0^{\ell_0}=1+Np\O(-2+Np\O)\equiv 1\pmod{Np^{\mu_0}\o_{K_3}}.
\end{equation*}
Since $-2+Np\O\not\in \mathfrak{p}$ for any prime ideal $\mathfrak{p}$ dividing $p\o_{K_3}$,
$\O$ belongs to $p^{\mu_0-1}\o_{K_3}$ and $\varepsilon_0^{{\ell_0}/{2}}\equiv -1\pmod{Np^{\mu_0}\o_K}$.
If $\varepsilon_0^{{\ell_0}/{2}}\equiv -1\pmod{Np^{\mu}\o_K}$ for some $\mu>\mu_0$, then $\varepsilon_0^{\ell_0}\equiv 1\pmod{Np^{\mu}\o_K}$, which contradicts the maximality of $\mu_0$.
\par
On the other hand, one can show by utilizing the idea in the proof of Lemma \ref{H-order} that $S_{N,p,\mu}\cap\mathcal{O}_K^\times\subset\o_{K_3}^\times=\{\pm\varepsilon_0^n~|~n\in\mathbb{Z} \}$ for all $\mu\in\mathbb{Z}_{>0}$. 
Here we observe that we don't need the assumptions $N\neq 2$.
Since $H_{N,p,\mu}/S_{N,p,\mu+1}\subseteq S_{N,p,\mu}/S_{N,p,\mu+1}\cong(\mathbb{Z}/p\mathbb{Z})^4$, we attain
\begin{equation*}
\dim_{\mathbb{Z}/p\mathbb{Z}}\big(H_{N,p,\mu}/S_{N,p,\mu+1}\big)=0~\textrm{or}~1.
\end{equation*}
It then follows from the above claim that $H_{N,p,\mu}/S_{N,p,\mu+1}=0$ if $\mu<\mu_0$.
Now, assume $\mu\geq\mu_0$.
Observe that $H_{N,p,\mu_0}/S_{N,p,\mu_0+1}=\langle\varepsilon_0^{\ell_0}S_{N,p,\mu_0+1}\rangle\cong\mathbb{Z}/p\mathbb{Z}$ because $\gcd(2,p)=1$.
For a positive integer $m$,
\begin{equation*}
(\varepsilon_0^{\ell_0})^m=1+mNp^{\mu_0}(\alpha_0+\beta_0\sqrt{d_1d_2})+\sum_{i=2}^m
\begin{pmatrix}
m\\
i
\end{pmatrix}
\big\{Np^{\mu_0}(\alpha_0+\beta_0\sqrt{d_1d_2}) \big\}^i.
\end{equation*}
Since $\alpha_0+\beta_0\sqrt{d_1d_2}\not\in p\o_{K_3}$ and 
\begin{equation*}
\sum_{i=2}^p
\begin{pmatrix}
p\\
i
\end{pmatrix}
\big\{Np^{\mu_0}(\alpha_0+\beta_0\sqrt{d_1d_2}) \big\}^i\in Np^{\mu_0+2}\o_{K_3},
\end{equation*}
$m=p$ is the smallest positive integer satisfying $(\varepsilon_0^{\ell_0})^m\in H_{N,p,\mu_0+1}\setminus S_{N,p,\mu_0+2}$.
In a similar fashion, one can prove by using the induction that $m=p^{\mu-\mu_0}$ is the smallest positive integer such that $(\varepsilon_0^{\ell_0})^m\in H_{N,p,\mu}\setminus S_{N,p,\mu+1}$.
Therefore, we conclude
\begin{equation*}
H_{N,p,\mu}/S_{N,p,\mu+1}=\langle\varepsilon_0^{\ell_0 p^{\mu-\mu_0}}S_{N,p,\mu+1} \rangle\cong\mathbb{Z}/p\mathbb{Z}.
\end{equation*}

\end{proof}

\begin{corollary}\label{ray class order2}
With the assumption as in Lemma \ref{H-order2}, for a positive integer $\mu$
\begin{eqnarray*}
\big[K_{(Np^{\mu+1})}:K_{(Np^\mu)}\big]=\left\{
\begin{array}{ll}
p^4 & \textrm{if $\mu<\mu_0$}\\
p^3& \textrm{if $\mu \geq\mu_0$},
\end{array}\right.
\end{eqnarray*}
\end{corollary}
\begin{proof}
It is immediate from (\ref{Galois group}), Lemma \ref{S-order2} and \ref{H-order2}.
\end{proof}

For a positive integer $\mu$ and $i=1,2$, let
\begin{equation*}
W_{N,p,\mu}^{1,2}=\{\O\in S_{N,p,\mu}~|~ N_{K/K_i}(\O)\equiv  1\pmod{Np^{\mu+1}\o_{K_i}}~~\textrm{for $i=1,2$} \}
\end{equation*}
so as to get $\ker(\widetilde{\V_{N,p,\mu}^{1,2}})=W_{N,p,\mu}^{1,2}/H_{N,p,\mu}$.

\begin{theorem}\label{K_{1,2}-2}
Let $N$ be a positive integer and $p$ be an odd prime not dividing $N$.
Assume that $K_1,K_2\neq \mathbb{Q}(\sqrt{-1}),\mathbb{Q}(\sqrt{-3})$.
Then for a positive integer $\mu$, we derive
\begin{equation*}
\big[K_{(Np^{\mu+1})}:\widetilde{K_{N,p,\mu}^{1,2}}\big] =\left\{
\begin{array}{ll}
p& \textrm{if $\mu<\mu_0$}\vspace{0.1cm}\\
1& \textrm{if $\mu\geq\mu_0$},
\end{array}\right.
\end{equation*}
\end{theorem}

\begin{proof}
For any coset $\alpha S_{N,p,\mu+1}$ in $S_{N,p,\mu}/S_{N,p,\mu+1}$ we can choose $\O\in\o_K\cap S_{N,p,\mu}$ such that $\alpha S_{N,p,\mu+1}=\O S_{N,p,\mu+1}$.
And, we write $\O=1+(Np^\mu/2)\big(a+b\sqrt{-d_1}+c\sqrt{-d_2}+d\sqrt{d_1d_2}\big)$ 
with $a,b,c,d\in\mathbb{Z}$ satisfying $a\equiv b\pmod{2}$, $c\equiv d\pmod{2}$.
\par
First, we consider the group $W_{N,p,\mu}^{1,2}/S_{N,p,\mu+1}$.
If $\O\in\o_K\cap W_{N,p,\mu}^{1,2}$ , then we have
\begin{equation*}
\begin{array}{rll}
N_{K/K_1}(\O)&\equiv& 1+Np^\mu(a+b\sqrt{-d_1})\equiv 1 \pmod{Np^{\mu+1}\o_{K_1}},\vspace{0.1cm}\\
N_{K/K_2}(\O)&\equiv& 1+Np^\mu(a+c\sqrt{-d_2})\equiv 1 \pmod{Np^{\mu+1}\o_{K_2}}.
\end{array}
\end{equation*}
Since $a+b\sqrt{-d_1}=a+b+2b\big(\frac{-1+\sqrt{-d_1}}{2}\big)\in p\o_{K_1}$ and $a+c\sqrt{-d_2}\in p\o_{K_2}$, $p$ divides $a,b,c$.
Thus we claim that 
\begin{equation*}
W_{N,p,\mu}^{1,2}/S_{N,p,\mu+1}=\left\langle (1+Np^\mu\sqrt{d_1d_2})S_{N,p,\mu+1}\right\rangle\cong\mathbb{Z}/p\mathbb{Z}, 
\end{equation*}
and by Lemma \ref{H-order2} we get
\begin{equation*}
\big[K_{(Np^{\mu+1})}:\widetilde{K_{N,p,\mu}^{1,2}}\big] =\frac{\big|W_{N,p,\mu}^{1,2}/S_{N,p,\mu+1}\big|}{\big|H_{N,p,\mu}/S_{N,p,\mu+1}\big|}=\left\{
\begin{array}{ll}
p& \textrm{if $\mu<\mu_0$}\vspace{0.1cm}\\
1& \textrm{if $\mu\geq\mu_0$}.
\end{array}\right.
\end{equation*}

\end{proof}

On the other hand, let $\F_{\mu,i}$ and $C_{\mu,i}$ be as in Section \ref{(I)}.
\begin{corollary}\label{level-generator}
With the notations and assumptions as above, if $\mu\geq \mu_0$, then for any positive integers $n_1$, $n_2$,
the value
\begin{equation*}
\prod_{i=1}^2 g_{\F_{\mu+1,i}}(C_{\mu+1,i})^{n_i}
\end{equation*}
generates $K_{(Np^{\mu+1})}$ over $K_{(Np^{\mu_0})}$.
\end{corollary}
\begin{proof}
By Proposition \ref{imaginary generator}, Lemma \ref{widetilde} and Theorem \ref{K_{1,2}-2} we obtain
\begin{equation*}
K_{(Np^{\mu+1})}=K_{(Np^\mu)}(K_1)_{(Np^{\mu+1})}(K_2)_{(Np^{\mu+1})}=K_{(Np^{\mu_0})}(K_1)_{(Np^{\mu+1})}(K_2)_{(Np^{\mu+1})}=K_{(Np^{\mu_0})}\left(\gamma_{1}, \gamma_{2}\right),
\end{equation*}
where $\gamma_{i}=g_{\F_{\mu+1,i}}(C_{\mu+1,i})^{n_i}$ for $i=1,2$.
Since the only element of $\mathrm{Gal}(K_{(Np^{\mu+1})}/K_{(Np^{\mu_0})})$ fixing the value $\gamma_{1}\gamma_{2}$ is the identity, we get the conclusion (\cite[Theorem 3.5 and Remark 3.6]{Jung}).

\end{proof}

\begin{example}\label{example2}
Let $K=\mathbb{Q}\left(\sqrt{-15},\sqrt{-26}\right)$, $N=5$ and $p=37$ as in Example \ref{example1} (ii).
Then $\varepsilon_0=79+4\sqrt{390}$, $\ell_0=190$ and $\mu_0=1$.
Hence by Corollary \ref{level-generator} and (\ref{example 1-1}) we get that for any positive integers $n_1$, $n_2$ and $\mu$,
\begin{equation*}
K_{(5\cdot 37^{\mu})}=
K_{(5)}\Big(
g_{\left[\begin{smallmatrix}0\\1/(5\cdot 37^{\mu})\end{smallmatrix}\right]}(\theta_{1})^{60\cdot 37^{\mu}n_1}
g_{\left[\begin{smallmatrix}0\\1/(5\cdot 37^{\mu})\end{smallmatrix}\right]}(\theta_{2})^{60\cdot 37^{\mu}n_2}
\Big)
\end{equation*}
where $\theta_1=({-1+\sqrt{-15}})/{2}$ and $\theta_2=\sqrt{-26}$.

\end{example}

\bibliographystyle{amsplain}

\address{
Ja Kyung Koo\\
Department of Mathematical Sciences \\
KAIST \\
Daejeon 34141 \\
Republic of Korea} {jkkoo@math.kaist.ac.kr}
\address{
Dong Sung Yoon\\
Department of Mathematical Sciences \\
KAIST \\
Daejeon 34141 \\
Republic of Korea} {math\_dsyoon@kaist.ac.kr}

\end{document}